\documentclass{amsart}
\usepackage{graphicx}
\usepackage{amssymb}
\usepackage{graphicx}
\usepackage{amsfonts}
\usepackage{amsmath}
\usepackage{esint}
\usepackage{mathabx}
\usepackage{mathtools}
\usepackage{nccmath}
\usepackage[colorlinks]{hyperref}

\newtheorem{thm}{Theorem}[section]

\newtheorem{lem}[thm]{Lemma}
\newtheorem{prop}[thm]{Proposition}
\theoremstyle{definition}

\theoremstyle{remark}
\newtheorem{rem}[thm]{Remark}

\begin{document}

\title[Pointwise Schauder for Optimal Transport]{Pointwise Schauder estimates for optimal transport maps of rough densities}

\author{Arghya Rakshit}
\address{Department of Mathematics, UC Irvine}
\email{\tt arakshit@uci.edu}

\date{\today}

\begin{abstract}
We prove a pointwise $C^{2,\,\alpha}$ estimate for the potential of the optimal transport map in the case that the densities are only close to constant in a certain $L^p$ sense.
\end{abstract}
\maketitle



\section{Introduction}\label{Pri}

The theory of optimal transport has been of mathematical interest for centuries. One way to attack the regularity theory for optimal transport is by Caffarelli's approach using the theory of Monge-Amp\`{e}re equations (see \cite{LC1,LC2,LC3,LC4}). In this paper,  we take a variational approach that was developed by Goldman and Otto (see \cite{GO}) in  the context of optimal transport. The general viewpoint of regularity of the optimal transport map using variational techniques has been used in \cite{GHO,GO,
LK,KO,OPR}. The variational framework employed here parallels the foundational ideas of De Giorgi to study minimal surfaces (see \cite{EG}).

We want to prove a pointwise Schauder estimate for the potential of the optimal transport map. Pointwise Schauder estimates have appeared at least since the work of Caffarelli on fully nonlinear equations (see \cite[Theorem~8.1]{CC}, \cite[Theorem~2 and Theorem~3]{LC5}). His regularity theory of Monge-Amp\`{e}re equations is done in a pointwise sense (see \cite{LC1}, \cite{LC4}, \cite{LC5}). Savin studied the boundary regularity of the Monge-Amp\`{e}re equation in a pointwise sense (see \cite{OS}).  In this paper, we prove a pointwise regularity estimate for the potential of the optimal transport map. One can view this as a finer version of Schauder estimates in the sense that we can get a Taylor expansion at a point assuming only appropriate regularity hypotheses at that point.\\

We begin by fixing some notation.  Let $B_r$ denote an open ball of radius $r$ in $\mathbb{R}^n$ centered at the origin, where $n\geq 2$.  For any measurable set $A\subset \mathbb{R}^n$, let $|A|$ denote its Lebesgue measure. We define 
$$\fint_A f:=\frac{1}{|A|}\int_A f$$
for any integrable function $f$ and any measurable set $A\subset\mathbb{R}^n$.
For any mass density $\rho:\mathbb{R}^n \to \mathbb{R}$ and $r > 0$, define $||\rho||_{p,r}$ as follows:
$$||\rho||_{p,r}:=\left(\fint_{B_r}|\rho-1|^p\right)^{\frac{1}{p}}.$$   
In this paper, we fix $p>n$. We also fix $\alpha \in (0,1)$.

Our setting is as follows. Suppose $\rho_0$ is a mass density supported in $B_1\subset \mathbb{R}^n$ and $\int \rho_0=|B_1|$. Let $\rho_1=\chi_{E}$ with $|E|=|B_1|$, where $\chi_E$ denotes the characteristic function of $E$. Assume $C_0^{-1} \leq \rho_0\leq C_0$ for some $C_0>0$ and $E$ contains $B_{1/2}$.
Let $T$ be the optimal transport map with quadratic cost that takes $\rho_0$ to $\rho_1$.
By a result of Brenier (see \cite{Brenier}), such a map $T$ exists. Moreover, $T=\nabla u$ for some convex function $u$. 
We refer to $u$ as the potential of the optimal transport map. We define excess energy $\mathcal{E}$ as follows:
$$\mathcal{E}(\rho_0,\rho_1, T, R):=\frac{1}{R^2}\fint_{B_R}|T-x|^2\rho_0\,dx.$$
A constant is called universal if it depends only on $n,\,p,\,\alpha,$ and $C_0$. We say that positive quantities $a$ and $b$ satisfy $a \lesssim b$ if $\frac{a}{b}$ is bounded above by a universal constant. By the notation $a\ll 1$, we mean there exists some small universal constant $\delta>0$ , such that if $a\leq \delta$, then the conclusions hold.

Our aim is to produce a pointwise quadratic approximation for $u$:
\begin{thm}\label{Main}
Let $\rho_0$ and $\rho_1$ be two densities described as above, and let $T=\nabla u$ be the optimal transport map that takes $\rho_0$ to $\rho_1$. There exists a universal constant
$\epsilon_1>0$ such that if 
 $\rho_0$ satisfies the following measure-theoretic inequality:
\begin{equation}\label{decay}
||\rho_0||_{p,r} \leq \sqrt{\epsilon_1}r^\alpha ,
\end{equation}
for all $0 < r < 1$, and 

    $$\mathcal{E}(\rho_0,\rho_1,T, 1)+||\rho_0||_{p,1}^2+||\rho_1||_{p,1}^2 \leq \epsilon_1, $$
then there exists a quadratic polynomial $Q$ such that, for all $0<r<\frac{1}{2}$, $$\fint_{B_r}|u-Q|^2\lesssim r^{2(2+\alpha)}.$$ Moreover, for all $0<r<\frac{1}{2}$,
\begin{equation}\label{mainapp}
||u-Q||_{L^\infty(B_{\frac{r}{2}})}\lesssim r^{2+\frac{4\alpha}{4+n}}.
\end{equation}
\end{thm}
Roughly speaking, Theorem \ref{Main} states that even when one of the densities is rough (e.g., has sharp spikes and valleys at small scales), as long as it behaves in a H\"{o}lder fashion in an appropriate integral sense, the potential enjoys a pointwise second-order Taylor expansion. Equation \eqref{mainapp} is sharp in the sense that one cannot get a better $L^\infty$ estimate starting from the $L^2$ estimate we have (see Section \ref{sharp}).\\

When $\rho_0$ is $C^\alpha$, the regularity of the corresponding Monge-Amp\`{e}re equations has already been established (see \cite{LC4, FJM}). In fact, one can generalize the proof of \cite{FJM} to prove Theorem \ref{Main} by using the ABP estimate (see \cite[Lemma~1.2]{LC3}) instead of the comparison principle. Our main result, Theorem \ref{Main}, can be viewed as an $\epsilon$-regularity theorem. As a general principle, $\epsilon$-regularity results are closely related to and useful in proving partial regularity theorems. For partial regularity in the context of the Monge-Amp\`{e}re equation and optimal transport, see \cite{DF} and \cite{FK}. Goldman and Otto (see \cite{GO}) gave a variational proof of partial regularity for optimal
transport maps when the source and target densities are $C^\alpha$.\\


Our work is inspired by \cite{GHO} and \cite{GO}. The first step to prove Theorem \ref{Main} is to establish the harmonic approximation theorems (Theorem \ref{ev} and Theorem \ref{lag}). In fact, these harmonic approximation theorems have been proved in \cite[Theorem~1.4 and Theorem~1.6]{GHO}. For a somewhat simpler and more motivated proof, see \cite{KO}.  Although the harmonic approximation results we prove are not new, and we are working with a stronger hypothesis $p>n$, the techniques we use may still be of value. The stronger assumption allows us to significantly simplify the proof of the harmonic approximation compared to \cite{GHO}. Next, we compare and contrast our results and techniques with those of Goldman and Otto in \cite{GO}, and Goldman, Huesmann, and Otto in \cite{GHO}. To do so, we delve into technical details; readers unfamiliar with \cite{GO} and \cite{GHO} may wish to skip this part of the introduction on a first reading.

To prove Theorem \ref{ev}, we need to prove that the interpolation density $\rho$ (defined in the beginning of Section \ref{lemmas}) is close to the uniform density. In \cite{GO}, they establish $L^\infty$ closeness of $\rho$ to $1$ using the concavity of the map $\det^{1/n}$. In \cite{GHO}, they prove the corresponding result in the Wasserstein distance sense (see Lemma 2.10 in \cite{GHO}). The proof uses elliptic PDE, careful analysis techniques, and an appropriate choice of radius to obtain derivative bounds for the harmonic approximation. In our paper, we show that it is enough to have an $L^1$ closeness of $\rho$ to uniform density. To do so, we use different PDE techniques, such as Calderón-Zygmund theory. This yields a more streamlined and direct proof, at the cost of the stronger assumption $p > n$. The work of Goldman,  Huesmann, and Otto (see  \cite{GHO}) shows that this assumption is not required.  In Lemma 3.3, we show that $\rho$ is close to $1$ in an $L^1$ sense, and this proof is new. To do so, we use the concavity of $\det^{1/n}$, a displacement bound (Lemma \ref{lem1}), and some analytical tools, such as Chebyshev’s inequality. In the proof of harmonic approximation (Theorem \ref{ev}), we use Calderón-Zygmund estimates to obtain a global $L^{\infty}$ bound for $\nabla\phi$, where $\phi$ is the harmonic approximation of the potential of the optimal transport map defined in \eqref{phi}. Additionally, since we are working in 
$L^p$ spaces, we must work with integral estimates rather than pointwise ones, which complicates some aspects of the analysis compared to \cite{GO}. 

Another key difference from \cite{GO} is that we assume \eqref{decay} only at a single point. This means we cannot use the Morrey-Campanato space theory as in \cite{GO} to directly get \eqref{mainapp}. Instead, we use some convex analysis (see Lemma \ref{pointwise}) to conclude Theorem \ref{Main}. To prove this lemma, we rely on the convexity of the potential of the optimal transport map, and we obtain a sharp $L^2$ to $L^\infty$ estimate.  This $L^2$ to $L^\infty$ type estimate is a general result about semi-convex functions and it can be used in many other contexts.   

The induction step in the proof of $\epsilon$-regularity is also quite challenging in our case. When we keep applying the one-step improvement in the proof of $\epsilon$-regularity, we must take a different path to control the source and target densities because we are working with $L^p$ norms (rather than $C^{\alpha}$ norms, as in \cite{GO}). Moreover, to exploit that $B_{1/2} \subset E$ we must carefully keep track of shifts of the target density that happen along the iteration. 

We now discuss our assumptions on $\rho_0$ and $\rho_1$. We assume that $\rho_0$ is an $L^p$ density that is bounded above, allowing for bounded spikes in the source density. This assumption is crucial for our proof of Theorem \ref{ev}. In particular, it is used to justify Lemma \ref{uniform} and inequality \eqref{un1}. It is unclear whether Theorem \ref{ev} can be proved using our technique without this assumption, although the theorem remains valid according to \cite[Theorem~1.4]{GHO} without this assumption. Their work also shows that the harmonic approximation theorem remains valid without assuming a lower bound on $\rho_0$. We have used this assumption in some key steps in the proof, for example in the proof of Lemma \ref{lem1} (which can be found in \cite{GO}) and in the proof of Theorem \ref{ev} where a lower bound on $\Tilde{\rho}$ is required (see equation \eqref{competetor}). The target density $\rho_1$ is assumed to be essentially constant. This assumption is important to run the iteration in the proof of Theorem \ref{epsilon} (see discussion in the previous paragraph). \\

This paper is organized as follows. In Section \eqref{counterexample}, we present an example in the plane that shows why it is important for the small constant $\epsilon_1$ to appear in equation \eqref{decay} of Theorem \ref{Main}. We then turn to the proof of the main theorem. We begin by proving some key lemmas that will be used later. In particular, we show that the interpolation density $\rho$ (see the beginning of Section \ref{lemmas} for its definition) is close to the uniform density in the $L^1$ sense. We also prove an $L^2$ to $L^\infty$ type estimate for convex functions. In Section \ref{harmonic} we find a harmonic function $\phi$ whose gradient well approximates our optimal transport map at a certain scale. Then, in Section \ref{ota}, we show that if $T$ is close to the identity at a certain scale, then after rescaling , it is even closer to the identity. We iterate this procedure infinitely many times to get the $L^2$ estimate in our main theorem.  Using this $L^2$ estimate and the semi-convex analysis mentioned earlier, we complete the proof of Theorem \ref{Main}.  Finally, in Section \ref{sharp}, we show that one cannot obtain better than pointwise $C^{2,\,\sim\frac{\alpha}{n}}$ regularity starting from the $L^2$ estimate we have.

\section{Counterexample when $\rho_0$ has $C^\alpha$ decay in dimension $n=2$}\label{counterexample}

Consider the non-radial function $u$ constructed in Section 4.2 of \cite{FJM}. This function $u$ satisfies the following homogeneity relation:
\begin{equation}\label{homogeneity}
    u(x,y)=\lambda^{-1-\gamma}u(\lambda x,\lambda^\gamma y)
\end{equation}
for $\gamma>1$, and is symmetric under reflection over both axes. It has been shown in \cite[Section~4.2]{FJM} that $\det(D^2u)=f$ is smooth away from the origin, strictly positive, bounded and Lipschitz on $\max\{|x|,|y|\}=1$. Moreover, $\det(D^2u)$ remains constant along the curves $y = m|x|^\gamma$ for all $m \in \mathbb{R}$. In addition, $u$ is not $C^2$ at the origin. Thus, \eqref{mainapp} does not hold for any quadratic polynomial $Q$. Fix $0<\alpha<1$ and $p\geq 1$, and let $\gamma=\frac{2}{2-\alpha}+p\alpha+1$. We show that
$$\left( 
\frac{1}{r^2}\int_{R_r} |f-a|^p\right)^{\frac{1}{p}} \lesssim r^\alpha$$ for all $r>0$ and some fixed $a > 0$. Multiplying $u$ by $a^{-1/2}$ and setting $\rho_0 = f/a$ and $\rho_1 = 1$, we see that a small value of $\epsilon_1$ is necessary in the main theorem.

For computational simplicity, we work in a square $R_r$ of side length $r$ centered at the origin, rather than a ball of radius $r$. As per the calculations in \cite{FJM}, for $y \neq 0$ and $t=|y|^{-\frac{1}{\gamma}}|x|$,
$$\det(D^2u)=\gamma^{-2}g''(t)\left((\gamma+1) g(t)+(\gamma-1)tg'(t )\right)+g'(t)^2,$$
where $g(t)=u(t,\pm 1).$

\begin{figure}
\hspace*{-2cm}                                                        
\includegraphics[scale=0.3]{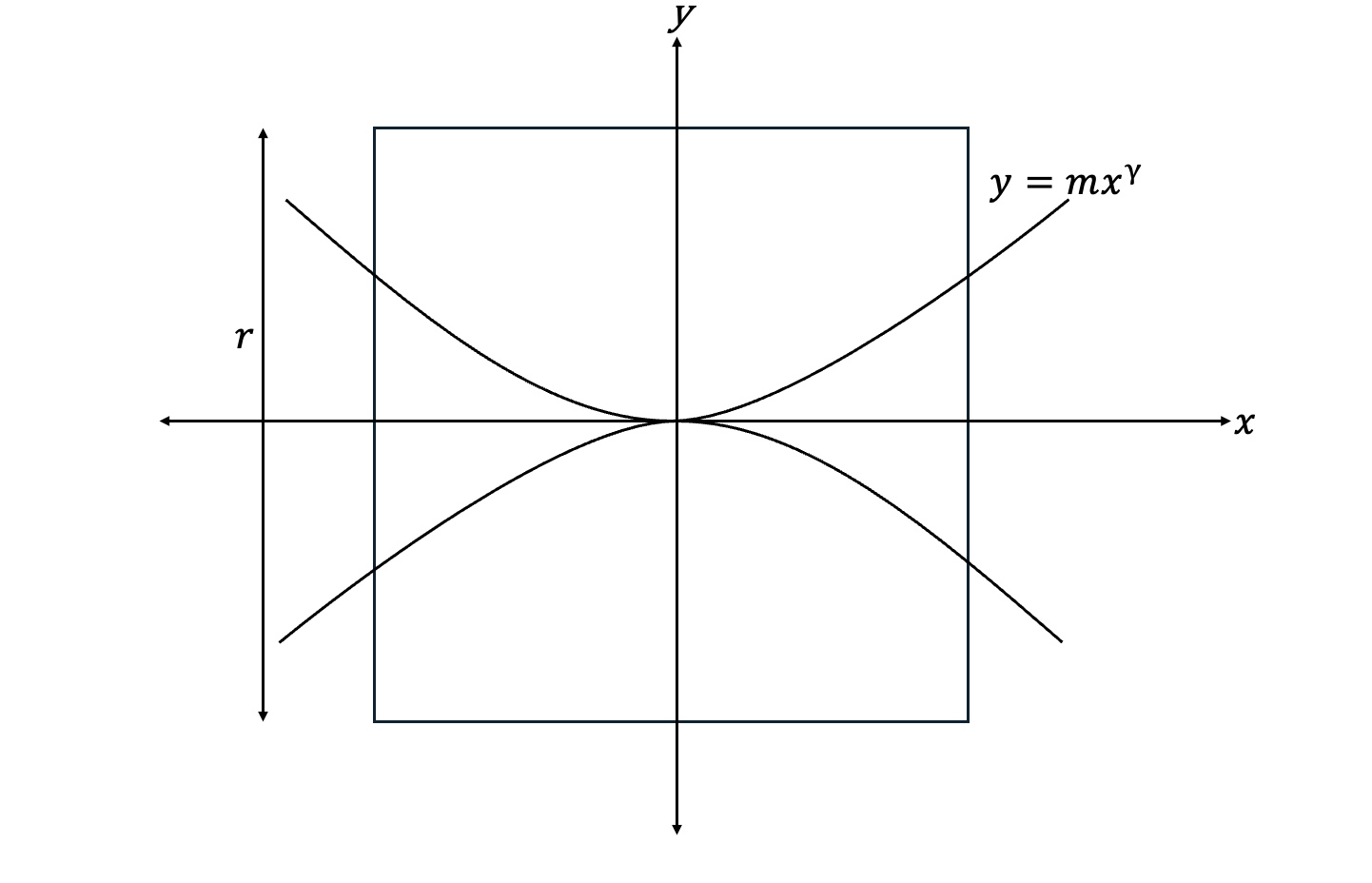}
\caption{Level sets for $\det D^2u(x,y)$}
\label{Fig1}
\end{figure}

By the homogeneity and symmetries of $u$, we have 
$$\det D^2u(x,y) = \det D^2u(|x||y|^{-\frac{1}{\gamma}},1)$$
for all $y \neq 0.$ We recall from \cite{FJM} that $g$ is equal to $c+dt^2$ for small $|t|$, and some constants $c, d > 0$.

Let $m > 0$ be a large constant to be chosen shortly. We let $A_r$ be the region in $R_r$ that either lies above $y=m|x|^\gamma$ or below $y=-m|x|^\gamma$. In $A_r$, we have $|t|<m^{-\frac{1}{\gamma}}$. Choose $m$ sufficiently large so that $g(t) = c + dt^2$ holds throughout the region $A_r$.

Now, for $(x,y) = (x,1)$, we have $t = x$, and thus,
\begin{equation*}
\begin{aligned}
    \det(D^2u(x,1))&=\gamma^{-2}2d\left((\gamma+1) (c+dx^2)+2d(\gamma-1)x^2\right)+4d^2x^2\\
    &=\frac{2cd(\gamma+1)}{\gamma^2}+\frac{2d^2}{\gamma^2}(2\gamma^2+3\gamma-1)x^2\\
    &=a+b x^2
\end{aligned}
\end{equation*}
  where $a:=\frac{2cd(\gamma+1)}{\gamma^2}$ and $b:=\frac{2d^2}{\gamma^2}(2\gamma^2+3\gamma-1)$. Hence, in the region $A_r$,
$$\det(D^2u(x,y))=\det(D^2u(|x||y|^{-\frac{1}{\gamma}},1))=a+b x^2|y|^{-\frac{2}{\gamma}}.$$ 
In the complement of the region $A_r$ within the rectangle $R_r$, $\det(D^2u)$ remains bounded. For small positive $r$, we have
\begin{equation*}
\begin{aligned}
\frac{1}{r^2}\int_{R_r} |f-a|^p&=\frac{1}{r^2}\int_{R_r\setminus A_r}|f-a|^p+\frac{1}{r^2}\int_{A_r}|f-a|^p\\
&\leq C \frac{1}{r^2}|A_r^c\cap R_r|+\frac{4}{r^2}\int_0^r\int_{mx^\gamma}^r b^p x^{2p} y^{-\frac{2p}{\gamma}}\\
&\leq C \left(r^{\gamma-1} + r^{2p\left(1-\frac{1}{\gamma}\right)}
\right).
\end{aligned}
\end{equation*}

This concludes the estimate, yielding
$$\left( 
\frac{1}{r^2}\int_{R_r} |f-a|^p\right)^{\frac{1}{p}}\leq C(r^{2(1-\frac{1}{\gamma})}+r^{\frac{\gamma-1}{p}})\leq Cr^\alpha$$ for all $r > 0$.

\section{Key Lemmas}\label{lemmas}
In this section, we prove several lemmas that will be instrumental in establishing our main theorem. For $0\leq t\leq 1$ and $x\in \mathbb{R}^n$, define $$T_t(x):=tT(x)+(1-t)x.$$ Let $\rho(t):=T_{t_{\#}}(\rho_0)$ and $j(t):=T_t{_{\#}}\left[\left(T-\mathrm{Id}\right)\rho_0\right]$. Note that $(\rho,j)$ satisfies the continuity equation, that is, $$\partial_t\rho+\nabla\cdot j=0.$$ Since $T$ is the optimal transport map, the pair $(\rho,j)$ minimizes $\int\frac{1}{\rho}|j|^2$ among all choices $(\rho,j)$ such that $\rho(0)=\rho_0$, $\rho(1)=\rho_1$ and the pair satisfies the continuity equation (for a proof, see Theorem 8.1 in \cite{CV}). For a more complete discussion on the definition of $\rho$ and $j$, look at section 3.1 in \cite{GO}.\\
We now introduce the notation $\gamma:= ||\rho_0||_{p,1}+||\rho_1||_{p,1}$, which will appear frequently throughout the paper. 

\begin{rem}
 Since $\rho_1$ is not taken to be the uniform density on all of  $B_1$, we note that $||\rho_1||_{p,1}\neq 0$.
\end{rem}

 We first prove an $L^\infty$ bound on the optimal transport map $T$. The proof of this remains the same as Lemma 3.1 in \cite{GO}.

\begin{lem}\label{lem1}
Let $T$ be the optimal transport map from $\rho_0$ to $\rho_1$. Assume $\rho_0\geq \frac{1}{C_0}$. Assume $\mathcal{E}\ll 1$. Then 
$\sup_{B_\frac{3}{4}}|T-x|+\sup_{B_\frac{3}{4}}|T^{-1}-x|\lesssim \mathcal{E}^{\frac{1}{n+2}}$. Moreover for all $t\in [0,1]$, we have $T_{t}(B_\frac{1}{8})\subset B_\frac{3}{16}$ and $T_t^{-1}(B_\frac{1}{2})\subset B_\frac{3}{4}$.
\end{lem}

In the next lemma, we prove that $\rho$ cannot be arbitrarily far away from $1$. The idea is to use concavity of $\det^\frac{1}{n}$.

\begin{lem}\label{uniform}
Assume $\mathcal{E}+\gamma^2\ll 1$. Then $\rho(t)\leq  C$ for some universal constant $C>0$.    
\end{lem}

 
\begin{proof}

For every $0<t<1$, $T_t$ has a well defined inverse a.e.\ (see Theorem 8.1 in \cite{CV}). This means for $x\in B_1$, we have $$\rho(t,x)=\frac{\rho_0(T_t^{-1}(x))}{\det \nabla T_t(T_t^{-1}(x))}.$$ 
 Concavity of $\det^{\frac{1}{n}}$ on symmetric matrices gives us
$\det \nabla T_t(T_t^{-1}(x))\geq (\det \nabla T(T_t^{-1})(x))^t$.  We also know $$\det \nabla T(T_t^{-1}(x))=\frac{\rho_0(T_t^{-1}(x))}{\rho_1(T(T_t^{-1}(x)))}.$$
Conclude that
\begin{equation*}
\rho(t,x)\leq \frac{\rho_0(T_t^{-1}(x))}{(\det \nabla T(T_t^{-1}(x)))^t}\leq (\rho_0(T_t^{-1}(x)))^{1-t}(\rho_1(T(T_t^{-1}(x))))^t.
\end{equation*}
Since both $\rho_0$ and $\rho_1$ are bounded, the result follows.
\end{proof}

Next, we show that the interpolation density $\rho(t,x)$ is close to $1$ in the ball of radius $\frac{1}{2}$ in $L^1$ sense.
\begin{lem}\label{densityclose}
For $\mathcal{E}+\gamma^2 \ll 1$, we have, $$\int_0^1\int_{B_{\frac{1}{2}}}|\rho-1|dx\:dt\lesssim \mathcal{E}^\frac{1}{n+2}+\sqrt{\gamma}.$$
\end{lem}
\begin{proof}
First using concavity of $\log \det,$ conclude that $\det(\nabla T_t)\geq \rho_0^t.$ Now, 
\begin{equation*}
\begin{aligned}
\int_0^1\int_{B_{\frac{1}{2}}}|\rho-1|dx\;dt
&=\int_0^1\int_{B_{\frac{1}{2}}}\left|\frac{\rho_0(T_t^{-1}(x))}{\det(\nabla T_t(T_t^{-1}(x)))}-1\right|dx\;dt\\
&=\int_0^1\int_{T_t^{-1}(B_{\frac{1}{2}})}\left|\frac{\rho_0(z)}{\det(\nabla T_t(z))}-1\right|\det(\nabla T_t(z)) dz\;dt\\
&=\int_0^1\int_{T_t^{-1}(B_{\frac{1}{2}})}|\rho_0(z)-\det(\nabla T_t(z))|dz\;dt.
\end{aligned}
\end{equation*}
Thus, it remains to show that
$$\int_0^1\int_{T_t^{-1}(B_{\frac{1}{2}})}|\rho_0(x)-\det(\nabla T_t(x))|dx\;dt\lesssim \mathcal{E}^\frac{1}{n+2}+\sqrt{\gamma}.$$ By triangle inequality,
$$\int_0^1\int_{T_t^{-1}(B_{\frac{1}{2}})}|\rho_0(x)-\det(\nabla T_t(x))|dx\;dt$$
$$\leq \int_0^1\int_{T_t^{-1}(B_{\frac{1}{2}})}|\rho_0(x)-1|dx\;dt+\int_0^1\int_{T_t^{-1}(B_{\frac{1}{2}})}|\det(\nabla T_t(x))-1|dx\;dt. $$
As by the previous lemma, $T_t^{-1}(B_\frac{1}{2})\subset B_1$, the first term of the right-hand side is bounded by $C\gamma$ for some constant $C$ depending only on dimension. To bound the second term, we use that $\det(\nabla T_t)\geq \rho_0^t$ for all $t$:

$$\int_0^1\int_{T_t^{-1}(B_{\frac{1}{2}})}|\det(\nabla T_t(x))-1|dx\;dt$$

$$\leq \int_0^1\int_{A_t}\left(1-\det(\nabla T_t(x))\right)dx\;dt+\int_0^1\int_{A_t^c}\left|1-\det(\nabla T_t(x))\right|dx\;dt,$$
where $A_t:=T_t^{-1}(B_\frac{1}{2})\cap \{\rho_0^t\leq\det(\nabla T_t)\leq (1-\sqrt{\gamma})^t\}$ and $A_t^c$ be the complement of $A_t$ inside $T_t^{-1}(B_\frac{1}{2}).$ We show that measure of $A_t$ cannot be too big:
\begin{equation*}
\begin{aligned}
|A_t|&\leq T_t^{-1}(B_\frac{1}{2})\cap \{\rho_0\leq 1-\sqrt{\gamma}\}\\
&\leq T_t^{-1}(B_\frac{1}{2})\cap \{|\rho_0-1|\geq \sqrt{\gamma}\}\\
&\leq \gamma^{\frac{-p}{2}}\int_{T_t^{-1}(B_\frac{1}{2})}|\rho_0-1|^p dx\\
&\lesssim \gamma^{\frac{p}{2}}, 
\end{aligned}
\end{equation*}
where we have used Chebyshev's inequality for the third step above.
Hence, 
$$\int_0^1\int_{A_t}\left(1-\det(\nabla T_t(x))\right)dx\;dt \leq \int_0^1 |A_t|\leq \gamma^{\frac{p}{2}}\leq\sqrt{\gamma}.$$
To finish the proof, all we need to show is 
$$\int_0^1\int_{A_t^c}\left|1-\det(\nabla T_t(x)\right)|dx\;dt \lesssim \mathcal{E}^\frac{1}{n+2}+\sqrt{\gamma}.$$
Now,
\begin{equation*}
\begin{aligned}
&\int_0^1\int_{A_t^c}\left|1-\det(\nabla T_t(x)\right)|dx\;dt \\
& \leq\int_0^1\int_{A_t^c}(\det(\nabla T_t(x))-1+2(1-(1-\sqrt{\gamma})^t)dx\;dt.
\end{aligned}
\end{equation*}
As $(1-(1-\sqrt{\gamma})^t)\leq \sqrt{\gamma}$ for all $t$, it is enough to show: 
$$\int_0^1\int_{A_t^c}(\det(\nabla T_t(x))-1)dx\;dt\lesssim \sqrt{\gamma}+\mathcal{E}^\frac{1}{n+2}.$$
\noindent
We have,
\begin{equation*}
\begin{aligned}
&\int_0^1\int_{A_t^c}(\det(\nabla T_t(x))-1)dx\;dt\\
&= \int_0^1\int_{A_t^c} \det(\nabla T_t(x)) dx\;dt- \int_0^1\int_{A_t^c} 1dx\;dt\\
&\leq \int_0^1\int_{T_t^{-1}(B_\frac{1}{2})}\det(\nabla T_t(x))dx\;dt-|A_t^c|\\
&=\int\int_{B_{1/2}}dz\;dt-|A_t^c|\\
&=|B_\frac{1}{2}|-|A_t^c|.
\end{aligned}
\end{equation*}
We already know that $|A_t|\leq C\gamma^\frac{p}{2}$. Hence,
$$|A_t^c|=|T_t^{-1}(B_\frac{1}{2})|-|A_t|\geq |T_t^{-1}(B_\frac{1}{2})|-C\gamma^{\frac{p}{2}}\geq |B_{\frac{1}{2}-C\mathcal{E}^\frac{1}{n+2}}|-C\gamma^\frac{p}{2}$$
where we use Lemma \ref{lem1} in the last inequality.
 Thus, 
$$\int_0^1\int_{A_t^c}(\det(\nabla T_t(x))-1)dx\;dt\leq |B_\frac{1}{2}|-|B_{\frac{1}{2}-C\mathcal{E}^\frac{1}{n+2}}|+C\gamma^\frac{p}{2}$$
$$\lesssim \mathcal{E}^{\frac{1}{n+2}}+\gamma^{\frac{p}{2}}.$$
This completes the proof.
\end{proof}
Next, using convex analysis we want to show an $L^2$ to $L^\infty$ type estimate:
\begin{lem}\label{pointwise}
Let $Q(x)=|x|^2/2$.
If $\fint_{B_1}|\Tilde{u}(x)-Q(x)|^2\lesssim \delta^2$ and $\Tilde{u}$ is convex, then $|\Tilde{u}-Q|\lesssim \delta^{\mu}$ in $B_\frac{1}{2}$, where $\mu=\frac{4}{4+n}$.
\end{lem}
\begin{proof}
\begin{figure}[htp]
    \centering
    \includegraphics[width=12cm]{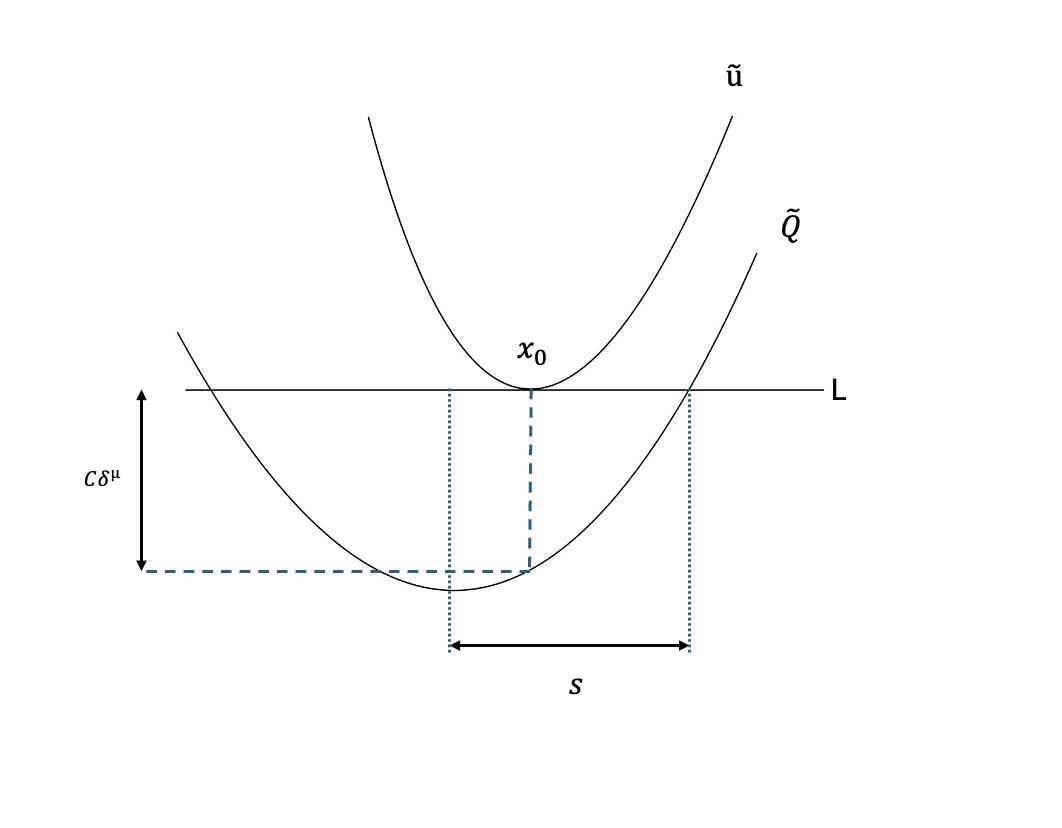}
    \caption{$u$ lies above $Q$}
    \label{fig1}
\end{figure}
Suppose, for contradiction, that the conclusion fails. We divide the proof into two cases. First let us assume $\Tilde{u}(x_0)-Q(x_0)\geq C \delta^\mu$, that is $\Tilde{u}$ lies above $Q$ (see Figure \ref{fig1}), for some large constant $C$ to be fixed later and for some $x_0\in B_\frac{1}{2}$. Up to subtracting a linear function we may assume $\nabla\Tilde{u}(x_0)=0$ (see Figure \ref{fig1}). Let $\tilde{Q}$ be $Q$ minus said linear function. Let $B:=\{x:\tilde{Q}(x)<\Tilde{u}(x_0)\}$  be the ball where $\tilde{Q}$ lies below the tangent of $\Tilde{u}$ at $x_0$.  Define $s$ to be the radius of the ball $B$. Because $\tilde{Q}=\frac{|x|^2}{2}$ up to adding a linear function, we have $s\geq C^{1/2}\delta^{\frac{\mu}{2}}$. Now, $$\int_{B_1}|\Tilde{u}-\tilde{Q}|^2 \geq \int_{B}|\Tilde{u}(x_0)-\tilde{Q}|^2\geq c(n)C^{\frac{n+4}{2}}\delta^{\frac{4+n}{2}\mu} = c(n)C^{\frac{n+4}{2}}\delta^2,$$
a contradiction for $C$ large.\\

In the second case, we assume  $Q(x_0)-\Tilde{u}(x_0)\geq C\delta^\mu$ ($Q$ lies above $\Tilde{u}$) for some large $C$ to be fixed later (see Figure \ref{fig2}). Let $\tilde{Q}$ be $Q$ plus an affine function chosen such that $\nabla \tilde{Q}(x_0)=0$. Add the same linear to $\tilde{u}$ and continue to denote it the same way. As $\fint_{B_1}|\Tilde{u}(x)-\tilde{Q}(x)|^2\lesssim \delta^2$ , we know $\Tilde{u}$ must cross the tangent hyperplane of $\tilde{Q}$ at $x_0$ at least at one point $x_1$ such that $|x_1-x_0|\leq \delta^\frac{\mu}{2}$. Indeed, if this was false, arguing in a similar way as in the first case and using the convexity of $\tilde{u}$ we have 
$$\int_{B_1} |\Tilde{u}-\tilde{Q}|^2\geq \int_{B_{\delta^\frac{\mu}{2}}(x_0)} \left(\tilde{Q}(x_0)-\Tilde{u}\right)^2 \geq c(n)C^2\delta^2,$$ a contradiction for $C$ large. \\
Consider the supporting hyperplane of $\Tilde{u}$ at $x_1$. Call it $L$. From the above we know that $S := |\nabla L| \geq C\delta^{\mu/2}$. It is elementary to show that $$\max (L - \tilde{Q}) = S^2/2 + \nabla L \cdot (x_0 - x_1) \geq S(S/2 - \delta^{\mu/2}) \geq S^2/4,$$ provided $C > 4$. Hence, we conclude
$$\max (L - \tilde{Q}) \geq C^2\delta^{\mu}/4.$$
Since $\tilde{u} \geq L$, taking $B = \{L > \tilde{Q}\}$ and arguing as in the first case, we are done.

\begin{figure}[htp]
\hspace*{-2cm}                                                        
\includegraphics[width=13.2cm]{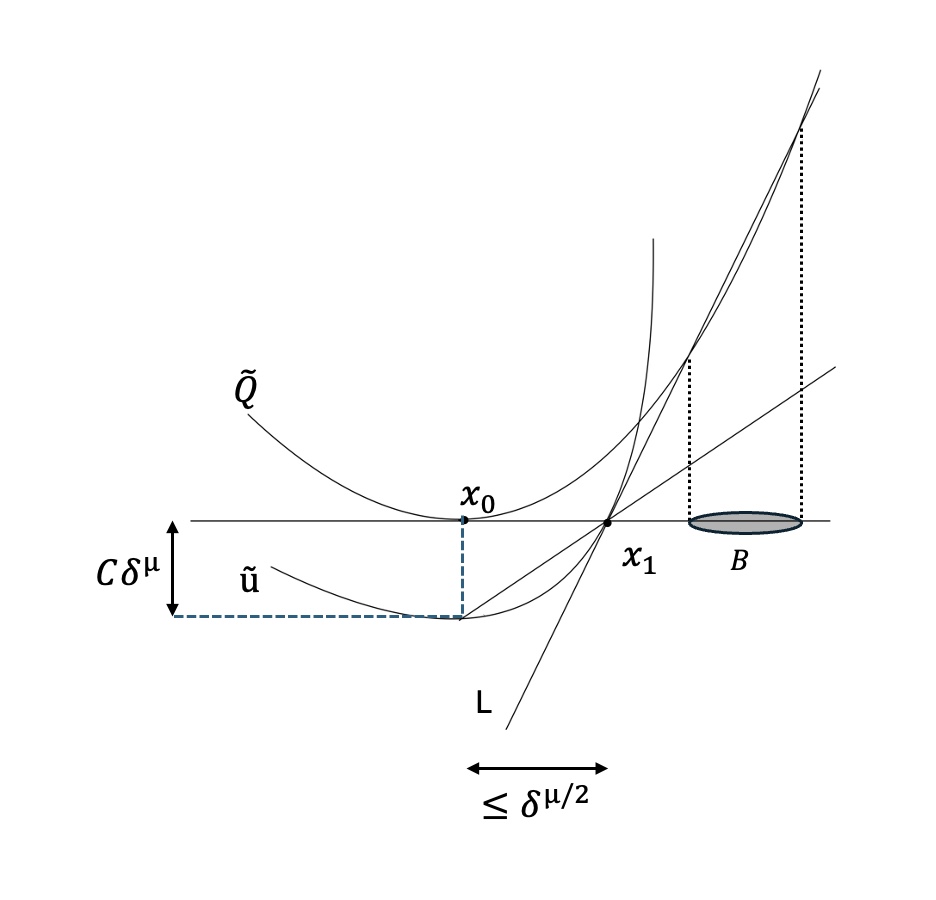}
\caption{$u$ lies below $Q$}
\label{fig2}
\end{figure}

\end{proof}

\section{Harmonic Approximation}\label{harmonic}
Recall  $\rho(t)=T_{t_{\#}}(\rho_0)$ and $j(t):=T_{t_{\#}}\left[\left(T-Id\right)\rho_0\right]$, where $T_t(x):=tT(x)+(1-t)x$. Here we prove that the derivation of the velocity vector field $v:=\frac{\partial j}{\partial\rho}$ can be well approximated by gradient of a harmonic function, if the excess energy is low and the densities $\rho_0$ and $\rho_1$ are close to the uniform density.

 \begin{thm}\label{ev}
Let $\mathcal{E}\ll 1$ and $\gamma\ll 1$. Then there exists a harmonic function $\phi$ such that   \begin{equation}\label{ev1}
    \int_{B_{\frac{1}{2}}}|\nabla\phi|^2\lesssim \int_0^1\int_{B_1}\frac{1}{\rho}|j|^2+\gamma^2,
\end{equation}
and 

\begin{equation}\label{ev2}
    \int_0^1\int_{B_{\frac{1}{2}}}\frac{1}{\rho}|j-\rho\nabla \phi|^2\lesssim \left(\int_0^1\int_{B_1}\frac{1}{\rho}|j|^2\right)^{\frac{n+2}{n+1}}+\gamma^2+\mathcal{E}^\frac{1}{n+2}\int_0^1\int_{B_1}\frac{1}{\rho}|j|^2.
\end{equation}
\end{thm}

\begin{proof}
Without loss of generality assume $\int\int\frac{1}{\rho}|j|^2\ll 1,$ as otherwise we can just take $\phi=0$.\\
First, we want to choose a good radius $R$ for which we can control the flux of $j$ through $\partial B_R$.\\ 
Use Fubini's theorem and $\rho \leq C$ for some $C>0$ as proved in Lemma \eqref{uniform}  to conclude 
$$\int_{R=\frac{1}{2}}^1\int_{\partial B_R}\int_0^1 |j|^2=\int_{B_1\setminus B_{\frac{1}{2}}}\int_0^1 |j|^2\lesssim \int_0^1 \int_{B_1} \frac{1}{\rho}|j|^2 .$$
Thus we can find a good radius $R$ with $\frac{1}{2}<R<1$ such that
\begin{equation}\label{un1}
\int_{\partial B_R}\int_0^1 |j|^2\lesssim \int_0^1\int_{B_1} \frac{1}{\rho} |j|^2 .
\end{equation}
Using the inequality for every $\zeta \in H^1 (B_R\times (0,1))$, we have
\begin{equation}\label{IBP}
\int_0^1\int_{B_R}\rho\partial_t\zeta +j\cdot \nabla \zeta=\int_0^1\int_{\partial B_R}\zeta f+\int_{B_R} \zeta(\cdot, 1)\rho_1-\zeta(\cdot, 0)\rho_0
\end{equation}
where $f:=j\cdot \nu$ be the normal component of $j$. 
For a proof, look at the paragraph after equation 3.20 in \cite{GO}. In fact in their proof,  they use an $L^2$ bound on $j$. This follows from equation \eqref{un1}. If we examine their proof carefully, we see that you can get equation \eqref{IBP} using just an $L^1$ bound on $j$. This would follow from Cauchy–Schwarz inequality and the fact that $\rho$ is bounded in $L^1$.\\

Now we want to define our harmonic approximation $\phi$, whose gradient well approximates $\frac{\partial j}{\partial\rho}$. Let $\phi$ be the solution to 
\begin{equation}\label{phi}
\begin{cases}
    -\Delta \phi =0\quad \text{in }\;\; B_R\\
    \frac{\partial\phi}{\partial\nu}=\overline{f}+ \frac{1}{\mathcal{H}^{d-1}(\partial B_R)}\int_{B_R}\delta \rho \quad \text{on} \;\; \partial B_R,
\end{cases}
\end{equation}
where $\overline{f}:=\int_0^1 f dt$ and $\delta\rho:=\rho_1-\rho_0$.
Applying \eqref{IBP} with $\zeta=1$ gives $\int_{B_R}\delta\rho =-\int_{\partial B_R}\overline{f}$. This shows that we can solve for $\phi$. \\
We now establish \eqref{ev1}. Apply equation 2.2 from Lemma 2.1 in \cite{GO} with radius $R$ replaced by $R\sim 1$ to get
\begin{equation}\label{unnamed1}
\begin{aligned}
\int_{B_{\frac{1}{2}}}|\nabla \phi|^2 & \leq \int_{B_{R}}|\nabla \phi|^2 \lesssim \int_{\partial B_R}|\overline{f}|^2+||\delta \rho||^2_{L^1(B_R)}\lesssim \int_{\partial B_R}|\overline{f}|^2+||\delta\rho||^2_{p,R} \\
& \leq  \int_{\partial B_R}\int_0^1|j|^2+||\delta\rho||^2_{p,R} 
\leq \int_0^1\int_{B_1}\frac{1}{\rho}|j|^2+\gamma^2.
\end{aligned}
\end{equation}
This proves \eqref{ev1}.
We can in fact conclude something even better if use equation 2.3 from Lemma 2.1 in \cite{GO}. We will have 
\begin{equation}\label{updated1}
||\nabla \phi||_{L^\infty(B_\frac{1}{2})}^2\leq \int_{B_R} |\nabla \phi|^2\lesssim \int_0^1\int_{B_1}\frac{1}{\rho}|j|^2+\gamma^2.
\end{equation}
Now let us define $\Tilde{\phi}$ and $\hat{\phi}$ as following:
\begin{equation*}
\begin{cases}
    -\Delta \Tilde{\phi}  =\delta \rho \quad \text{in }\;\; B_R\\
    \frac{\partial\Tilde{\phi}}{\partial\nu}=\overline{f} \quad \text{on} \;\; \partial B_R,
\end{cases}
\end{equation*}
and 
\begin{equation*}
\begin{cases}
    -\Delta \hat{\phi}  =\delta \rho \quad \text{in }\;\; B_R\\
    \frac{\partial\hat{\phi}}{\partial\nu}= - \frac{1}{\mathcal{H}^{d-1}(\partial B_R)}\int_{B_R}\delta \rho \quad \text{on} \;\; \partial B_R.
\end{cases}
\end{equation*}
Note that $\Tilde{\phi}=\phi+\hat{\phi}$. We will show that it suffices to prove
\begin{equation}\label{phitil1}
    \int_0^1\int_{B_R}\frac{1}{\rho}|j-\rho\nabla\Tilde{\phi}|^2\lesssim \left(\int_0^1\int_{B_1}\frac{1}{\rho}|j|^2\right)^{\frac{n+2}{n+1}}+\gamma^2+\mathcal{E}^\frac{1}{n+2}\int_0^1\int_{B_1}\frac{1}{\rho}|j|^2,
\end{equation}
and
\begin{equation}\label{phitil2}
    ||\nabla \Tilde{\phi}||^2_{L^\infty(B_\frac{1}{2})}\lesssim \int_{B_R}|\nabla \Tilde{\phi}|^2\lesssim \int_0^1\int_{B_R}\frac{1}{\rho}|j|^2+\gamma^2.
\end{equation}
Moreover, for $0<r\ll 1$, if we define $A_r:=B_R\setminus B_{R(1-r)}$, then
\begin{equation}\label{annulus}
    \int_{A_r}|\nabla\Tilde{\phi}|^2\lesssim r\left(\int_{\partial B_R}|\overline{f}|^2+\gamma^2\right).
\end{equation}

Let's now establish \eqref{phitil2}. For that we claim that
\begin{equation}\label{claim}
    ||\nabla\hat{\phi}||_{L^\infty(B_1)}\lesssim \gamma.
\end{equation}
First let us prove this claim.  Using global Calderón–Zygmund estimate(e.g. refer to \cite{FY}), we get $||D^2\hat{\phi}||_{L^p(B_1)}\lesssim \gamma$. By Poincaré inequality, we have $$||\nabla \hat{\phi}-\fint_{B_1}\nabla\hat{\phi}||_{L^p(B_1)}\leq ||D^2\hat{\phi}||_{L^p(B_1)}\lesssim \gamma.$$
Conclude  $$||\nabla \hat{\phi}-\fint_{B_1}\nabla\hat{\phi}||_{W^{1,p}(B_1)}=||\nabla \hat{\phi}-\fint_{B_1}\nabla\hat{\phi}||_{L^p(B_1)}+||D^2\hat{\phi}||_{L^p(B_1)}\lesssim \gamma.$$
By Morrey embedding theorem we have $W^{1,p}(B_1)\subset C^\alpha (\overline{B_1})$. Thus, 
\begin{equation}\label{bdd}
    |\nabla\hat{\phi}-\fint_{B_1}\nabla\hat{\phi}|\leq C\gamma 
\end{equation}
pointwise for all $x\in \overline{B_1}$.\\
Suppose our claim was false. That means for any $C$ we can find $x_0\in \overline{B_1}$ such that $|\nabla \hat{\phi}(x_0)|\geq C \gamma$. This with \eqref{bdd} gives $|\fint_{B_1}\nabla\hat{\phi}|\geq C\gamma$ for the same $C$. Conclude $|\nabla\hat{\phi}|\geq C\gamma$ at all points $x\in\overline{B_1}.$\\
We showed $\nabla\hat{\phi}$ blows up at all points on the boundary. Take $x_1\in \partial B_1$ such that $\hat{\phi}$ achieves its maximum on $\partial B_1$ at $x_1$. Hence, $\frac{\partial \hat{\phi}}{\partial T}=0$ where $T$ denotes the tangent direction. Thus, $\frac{\partial\hat{\phi}}{\partial\nu}$ must blow up. This is a contradiction to the Neumann Boundary condition we assumed in the definition of $\hat{\phi}$. This completes the proof of our claim.\\

This claim, along with \eqref{updated1} immediately gives us \eqref{phitil2} by triangle inequality. Similarly if we use equation 2.4 from \cite{GO} with equation \eqref{claim} above, we can get \eqref{annulus}.\\

Next we assume \eqref{phitil1} and prove \eqref{ev2}. Note that 
$$\int_0^1\int_{B_R} \frac{1}{\rho}|j-\rho\nabla\phi|^2=\int_0^1\int_{B_R}\frac{1}{\rho}|j-\rho\nabla\Tilde{\phi}+\rho\nabla\hat{\phi}|^2$$
$$\lesssim \int_0^1\int_{B_R} \frac{1}{\rho}|j-\rho\nabla\Tilde{\phi}|^2+\int_0^1\int_{B_R}\rho|\nabla\hat{\phi}|^2.$$
The first term is bounded by \eqref{phitil1}. Use Lemma \eqref{uniform} and our previous claim to bound the second term:
$$\int_0^1\int_{B_R}\rho|\nabla\hat{\phi}|^2\lesssim \gamma^2.$$
This completes the proof, except we have not proved equation \eqref{phitil1}. \\

From now on, up to rescaling, assume $R=\frac{1}{2}$. Here we prove that 
\begin{equation}\label{inequality1}
    \int_{0}^1\int_{B_{\frac{1}{2}}}\frac{1}{\rho}|j-\rho\nabla\Tilde{\phi}|^2\lesssim \int_0^1\int_{B_{\frac{1}{2}}}\frac{1}{\rho}|j|^2-  \int_0^1\int_{B_{\frac{1}{2}}} (2-\rho)|\nabla \Tilde{\phi}|^2.
\end{equation}
We have 
\begin{equation*}
\begin{aligned}
\frac{1}{2} \int_0^1\int_{B_{\frac{1}{2}}}\frac{1}{\rho}|j-\rho\nabla\Tilde{\phi}|^2 &= \frac{1}{2} \int_0^1\int_{B_{\frac{1}{2}}}\frac{1}{\rho}|j|^2- \int_0^1\int_{B_{\frac{1}{2}}}j\cdot \nabla \Tilde{\phi}+\frac{1}{2} \int_0^1  \int_{B_{\frac{1}{2}}}\rho |\nabla \Tilde{\phi}|^2 \\
&=\frac{1}{2} \int_0^1\int_{B_{\frac{1}{2}}}\frac{1}{\rho}|j|^2- \int_0^1\int_{B_{\frac{1}{2}}}\left(1-\frac{\rho}{2}\right)|\nabla \Tilde{\phi}|^2\\ & - \int_0^1\int_{B_{\frac{1}{2}}} (j-\nabla\Tilde{\phi})\cdot \nabla\Tilde{\phi}.
\end{aligned}
\end{equation*}
Note that if we put $\zeta=\Tilde{\phi}$ in \eqref{IBP}, we have the last term to be zero. Hence we get \eqref{inequality1}.\\

In the final part of this proof, we establish that 
\begin{equation}\label{inequality3}
\int_0^1\int_{B_{\frac{1}{2}}}\frac{1}{\rho}|j|^2-\int_{B_{\frac{1}{2}}}|\nabla \Tilde{\phi}|^2\lesssim \left(\int_0^1\int_{B_1}\frac{1}{\rho}|j|^2\right)^{\frac{n+2}{n+1}}+\gamma^2.
\end{equation}
First we show this \eqref{inequality3} along with  \eqref{inequality1} will give us \eqref{phitil1}:
\begin{equation}
\begin{aligned}
\int_0^1\int_{B_{\frac{1}{2}}}\frac{1}{\rho}|j-\rho\nabla\Tilde{\phi}|^2 &\lesssim \int_0^1\int_{B_{\frac{1}{2}}}\frac{1}{\rho}|j|^2-\int_0^1\int_{B_{\frac{1}{2}}} (2-\rho)|\nabla \Tilde{\phi}|^2\\
& \lesssim \left(\int_0^1\int_{B_1} \frac{1}{\rho}|j|^2\right)^{\frac{n+2}{n+1}}+\gamma^2-\int_0^1\int_{B_\frac{1}{2}}(1-\rho)|\nabla \Tilde{\phi}|^2
\end{aligned}
\end{equation}
Note that to get \eqref{phitil1} it is enough to prove 
$$\left|\int_0^1\int_{B_\frac{1}{2}}(1-\rho)|\nabla \Tilde{\phi}|^2\right| \lesssim 
\left(\int_0^1\int_{B_1}\frac{1}{\rho}|j|^2\right)^{\frac{n+2}{n+1}}+\gamma^2+\mathcal{E}^\frac{1}{n+2}\int\int\frac{1}{\rho}|j|^2.
$$
We do this by using \eqref{phitil2}, Lemma \ref{densityclose} and Young's inequality:
$$\left|\int_0^1\int_{B_\frac{1}{2}}(1-\rho)|\nabla \Tilde{\phi}|^2\right|\lesssim 
||\nabla\Tilde{\phi}||_{\infty}^2 \int\int_{B_\frac{1}{2}} |\rho-1| \lesssim \left(\sqrt{\gamma}+\mathcal{E}^{\frac{1}{n+2}}\right) \left(\int\int \frac{1}{\rho}|j|^2+\gamma^2\right)$$
$$
\lesssim \left(\int_0^1\int_{B_1}\frac{1}{\rho}|j|^2\right)^{\frac{n+2}{n+1}}+\gamma^2+\mathcal{E}^\frac{1}{n+2}\int\int\frac{1}{\rho}|j|^2.
$$
Thus, all we need to prove is \eqref{inequality3}. By minimality of $(\rho, j)$, it is enough to construct a competitor $(\Tilde{\rho}, \Tilde{j})$ that agrees with $(\rho,j)$ outside $B_{\frac{1}{2}}\times (0,1)$ and that satisfies the upper bound through \eqref{inequality3}.\\
We choose the following   $(\Tilde{\rho}, \Tilde{j})$
\begin{equation}\label{competetor}
      (\Tilde{\rho}, \Tilde{j})
      =
       \begin{cases}
           (t\rho_1+(1-t)\rho_0, \nabla \Tilde{\phi}) \quad \text{in}\;\; B_{\frac{1}{2}(1-r)}\times (0,1)\\
           (t\rho_1+(1-t)\rho_0+s, \nabla \Tilde{\phi}+q) \quad \text{in}\;\;A_r\times (0,1) ,
       \end{cases}
    \end{equation}
with $(s,q)\in \Lambda$ where $\Lambda$ is defined in Lemma 2.4 from \cite{GO} with $f$ replaced by $f-\overline{f}$ and the radius $1$ replaced by radius $\frac{1}{2}$. \\
Note that if $|s|\leq \frac{1}{2C_0}$, then $\Tilde{\rho}\geq \frac{1}{2C_0}$. In $B_{\frac{1}{2}(1-r)}$ we have $\Tilde{\rho}\geq \frac{1}{2C_0}$ as both $\rho_0$ and $\rho_1$ are greater than $\frac{1}{C_0}$. In $A_r$, we have $\Tilde{\rho}\geq \frac{1}{C_0}-\frac{1}{2C_0}\geq \frac{1}{2C_0}.$ 
Note that by definition of $\Tilde{\phi}$, our choices for $(\Tilde{\rho},\Tilde{j})$ are admissible.\\
By Lemma 2.4 in \cite{GO} (here we use a modified version of the Lemma: basically we take $|s|\leq \frac{1}{2C_0}$ instead of $|s|\leq \frac{1}{2}$ in the statement of the lemma), if 
\begin{equation}\label{un3}
    r\gg \left(\int_0^1\int_{B_{\frac{1}{2}}}(f-\overline{f})^2\right)^{\frac{1}{n+1}},
\end{equation}
we may choose $(s,q)\in \Lambda$ such that
\begin{equation}\label{inequality2}
\int_0^1\int_{A_r}|q|^2\lesssim r\int_0^1\int_{\partial B_{\frac{1}{2}}}(f-\overline{f})^2.
\end{equation}
By definition of $(\Tilde{\rho}, \Tilde{j}),$

$$\int_0^1\int_{B_{\frac{1}{2}}}\frac{1}{\Tilde{\rho}}|\Tilde{j}|^2-\int_{B_{\frac{1}{2}}}|\nabla \Tilde{\phi}|^2
$$
$$
\lesssim 
\int_0^1\int_{B_{\frac{1}{2}(1-r)}}\frac{1}{t\rho_1+(1-t)\rho_0}|\nabla(\Tilde{\phi})|^2
-\int_{B_{\frac{1}{2}(1-r)}} |\nabla(\Tilde{\phi})|^2+\int_0^1\int_{A_r}\frac{1}{\Tilde{\rho}}|\nabla\Tilde{\phi}+q|^2,
$$
where we have replaced the second term in the right-hand side by an integration on a smaller domain. \\
The first two terms on RHS can be estimated as follows
\begin{equation*}
\begin{aligned}
\int_0^1\int_{B_{\frac{1}{2}(1-r)}}\left(\frac{1}{t\rho_1+(1-t)\rho_0}-1\right)|\nabla\Tilde{\phi}|^2 &\lesssim  \int_0^1\int_{B_\frac{1}{2}(1-t)} (1-t)(1-\rho_0) |\nabla\Tilde{\phi}|^2\\
&\lesssim ||\rho_0-1||_{L^1} ||\nabla\Tilde{\phi}||_{L^\infty}^2 \\
&\lesssim ||\rho_0-1||_{L^p}  ||\nabla\Tilde{\phi}||_{L^\infty}^2\\
&\lesssim \gamma \left(\int\int\frac{1}{\rho}|j|^2+\gamma^2\right),
\end{aligned}
\end{equation*}
where in the first inequality we use $(1-t)\rho_0+t\rho_1$ is bounded from below in $B_{\frac{1}{2}(1-r)}$.\\
To estimate the last term use \eqref{inequality2}, \eqref{annulus} and \eqref{un1} respectively
\begin{equation*}
\begin{aligned}
\int_0^1\int_{A_r}\frac{1}{\Tilde{\rho}}|\nabla \Tilde{\phi}+q|^2&\lesssim 
\int_0^1\int_{A_r}|\nabla\Tilde{\phi}|^2+|q|^2 \\
& \lesssim \int_{A_r}|\nabla\Tilde{\phi}|^2+\int_0^1\int_{A_r}|q|^2\\
& \lesssim \int_{A_r}|\nabla\Tilde{\phi}|^2+r\int_0^1\int_{\partial B_{\frac{1}{2}}}(f-\overline{f})^2 \\
& \lesssim r\left(\int_{\partial B_{\frac{1}{2}}}\overline{f}^2+\gamma^2\right)+r\int_0^1\int_{\partial B_{\frac{1}{2}}}(f-\overline{f})^2 \\
&\lesssim r\left(\gamma^2+\int\int f^2\right)\\
& \lesssim r\left(\gamma^2+\int\int \frac{1}{\rho}|j|^2\right)
\end{aligned}
\end{equation*}
Choose $r$ to be large multiple of (order one though) $\left(\int_0^1\int_{B_{\frac{1}{2}}}(f-\overline{f})^2\right)^{\frac{1}{n+1}}$ to make sure \eqref{inequality2} holds. Note that,
\begin{equation*}
\begin{aligned}
\left(\int_0^1\int_{B_{\frac{1}{2}}}(f-\overline{f})^2\right)^\frac{1}{n+1} & \lesssim 2\left(\int_0^1\int_{\partial B_{\frac{1}{2}}} f^2+ \int_{\partial B_\frac{1}{2}}\overline{f}^2 \right)^{\frac{1}{n+1}} \lesssim \left(\int \int f^2\right)^\frac{1}{n+1}\\
&\lesssim \left(\int\int \frac{1}{\rho}|j|^2\right)^\frac{1}{n+1} ,
\end{aligned}
\end{equation*}
where we use \eqref{un1} to get the last inequality.\\
Combining last two calculations, if we denote $a=\int\int\frac{1}{\rho}|j|^2$, then by Young's inequality we have
\begin{equation*}
\begin{aligned}
\int_0^1\int_{B_{\frac{1}{2}}}\frac{1}{\Tilde{\rho}}|\Tilde{j}|^2-\int_{B_{\frac{1}{2}}}|\nabla \Tilde{\phi}|^2 & \lesssim \gamma(a+\gamma^2)+a^{\frac{1}{n+1}}(a+\gamma^2)\\
& \lesssim (\gamma+a^{\frac{1}{n+1}})(a+\gamma^2) =a^\frac{n+2}{n+1}+a\gamma+a^{\frac{1}{n+1}}\gamma^2+\gamma^2 \\
& \lesssim a^\frac{n+2}{n+1}+\gamma^2.
\end{aligned}
\end{equation*}
This finishes the proof of equation \eqref{inequality3} and we are done.
\end{proof}

Motivated by the theory of minimal surface equations, we next prove that $T-x$ is well approximated by $\nabla \phi$. 
\begin{lem}\label{lag}
Let $T$ be the optimal transport map that takes the density $\rho_0$ to $\rho_1$. Then there is a harmonic map $\phi$ in $B_\frac{1}{8}$ such that
\begin{equation}\label{lag1}
\int_{B_\frac{1}{8}}|T-x-\nabla\phi|^2\rho_0\lesssim \mathcal{E}^\frac{n+3}{n+2}+\gamma^2,
\end{equation}
and
\begin{equation}\label{lag2}
\int_{B_\frac{1}{8}}|\nabla\phi|^2\lesssim \mathcal{E} +\gamma^2.
\end{equation}
\end{lem}

\begin{proof}
First note that we may assume $\mathcal{E}+\gamma^2\ll 1$ since otherwise we can take $\phi=0$. We know that the velocity vector field $v:=\frac{\partial j}{\partial \rho}$ satisfies $v(T_t(x),t)=T(x)-x$ for a.e. $x\in supp(\rho_0)$ (see for example Theorem 8.1 in \cite{CV}). Thus, using \eqref{lem1} we have:

\begin{equation}\label{lv1}
\begin{aligned}
\int_0^1\int_{B_{\frac{1}{2}}}\frac{1}{\rho}|j|^2
&=\int_0^1\int_{B_\frac{1}{2}}|v|^2\;d\rho\\
&\lesssim \int_0^1\int_{T_t^{-1}(B_\frac{1}{2})}|T(x)-x|^2\rho_0 dx\\
&\lesssim \int_0^1\int_{B_1}|T-x|^2\rho_0\\
&=\mathcal{E}.
\end{aligned}
\end{equation}
By Lemma \ref{ev}, we can find a harmonic function $\phi$ in $B_\frac{1}{4}$ such that 

\begin{equation}\label{lv2}
\int_{B_{\frac{1}{4}}}|\nabla \phi|^2 \lesssim \int_0^1\int_{B_1}\frac{1}{\rho}|j|^2+\gamma^2,
\end{equation}
and 
\begin{equation}\label{lv3}
\begin{aligned}
\int_0^1\int_{B_{\frac{1}{4}}}\frac{1}{\rho}|j-\rho\nabla \phi|^2&\lesssim \left(\int_0^1\int_{B_\frac{1}{2}}\frac{1}{\rho}|j|^2\right)^{\frac{n+2}{n+1}}+\gamma^2+\mathcal{E}^\frac{1}{n+2}\int\int_{B_\frac{1}{2}}\frac{1}{\rho}|j|^2\\
&\lesssim \mathcal{E}^\frac{n+3}{n+2}+\gamma^2.
\end{aligned}
\end{equation}
For the second inequality above, we use \eqref{lv1}.
Using \eqref{lv1} and \eqref{lv2}, we get \eqref{lag2}.\\

We now prove \eqref{lag1}. By triangle inequality we have
\begin{equation}\label{lv4}
\begin{aligned}
\int_{B_\frac{1}{8}}|T-(x+\nabla\phi)|^2\rho_0\lesssim \int_0^1\int_{B_\frac{1}{8}}|T-x+\nabla\phi\circ T_t|^2\rho_0+\int_0^1\int_{B_\frac{1}{8}}|\nabla \phi-\nabla\phi\circ T_t|^2\rho_0.
\end{aligned}  
\end{equation}
Recall $|T_t(x)-x|\leq |T(x)-x|$ for all $t$ by definition of $T_t$ and we use this to estimate the second term on the right-hand side of \eqref{lv4}:

\begin{equation*}
\begin{aligned}
    \int_0^1\int_{B_\frac{1}{8}}|\nabla \phi-\nabla\phi\circ T_t|^2\rho_0
    & \lesssim \sup_{B_\frac{3}{16}}|D^2 \phi|^2\int_0^1\int_{B_\frac{1}{8}}|T_t-x|^2\rho_0\\
    &\lesssim \left( \int_{B_\frac{1}{4}}|\nabla \phi|^2\right) \mathcal{E}\\
    & \lesssim \mathcal{E}(\mathcal{E}+\gamma^2),
\end{aligned}
\end{equation*}
where we have used Lemma \eqref{lem1} and equation 2.3 in Lemma 2.1 in \cite{GO}.\\
Recall $v:=\frac{\partial j}{\partial \rho}$ satisfies $v(T_t(x),t)=T(x)-x$. This implies 
$$T(x)-x-\nabla\phi( T_t(x))=v(T_t(x),t)-\nabla\phi(T_t(x))$$
$$\implies T(x)-(x+\nabla\phi (T_t(x)))=\left(v(t,\cdot)-\nabla\phi\right)(T_t(x))$$ which holds for a.e. $x\in \operatorname{supp}(\rho_0)$. Let's now estimate the first term in the right-hand side of \eqref{lv4} using  \eqref{lv1} and \eqref{lv3}:

\begin{equation*}
\begin{aligned}
\int_0^1\int_{B_\frac{1}{8}}|T-x+\nabla\phi\circ T_t|^2\rho_0 &=\int_0^1\int_{T_t(B_\frac{1}{8})}|v-\nabla\phi|^2d\rho\\
& =\int_0^1\int_{T_t(B_\frac{1}{8})}\frac{1}{\rho}|j-\rho\nabla\phi|^2\\
& \lesssim \int_0^1\int_{B_\frac{1}{4}}\frac{1}{\rho}|j-\rho\nabla\phi|^2\\
&\lesssim \mathcal{E}^\frac{n+3}{n+2}+\gamma^2.
\end{aligned}
\end{equation*}
This completes the proof.
\end{proof}

\section{one-step improvement}\label{ota}
Next, we prove that if, at a certain scale $R$, $T$ is close to identity, that is if $$\mathcal{E}+\left(||\rho_0||_{p,R}^2+||\rho_1||_{p,R}^2\right)\ll 1,$$ then on scale $R\theta$, after an affine change of coordinates, it is even closer to identity by a geometric factor. Together with \eqref{lag1} from Lemma \ref{lag}, equation 2.3 from Lemma 2.1 in \cite{GO}, we complete the proof.

\begin{prop}\label{onestep}
For all $0<\alpha'<1$ there exists $\theta(n,\alpha,\alpha')>0$, $\epsilon(n,\alpha,\alpha')>0$ and $C_\theta(n,\alpha,\alpha')$ with the property that for all $R>0$ such that
\begin{equation}\label{osi}
  \mathcal{E}(\rho_0,\rho_1,T,R)+||\rho_0||_{p,R}^2+||\rho_1||_{p,R}^2\leq \epsilon
\end{equation}
there exists symmetric matrix $B\in SL_n(\mathbb{R})$ and a vector $b$ such that if we let $\hat{x}:=B^{-1}x$,
 $\hat{y}:= B(y-b)$,
$\hat{T}(\hat{x})=B(T(x)-b)$, $\hat{\rho_0}(\hat{x}):=\rho_0(x)$ and $\hat{\rho_1}(\hat{y})=\rho_1(y)$, then 
$$\mathcal{E}(\hat{\rho_0}, \hat{\rho_1},\hat{T},\theta R)\leq \theta^{2\alpha'}\mathcal{E}(\rho_0,\rho_1, T,R)+C_\theta \left(||\rho_0||_{p,R}^2+||\rho_1||_{p,R}^2\right).$$
Moreover, 
$$|B-Id|^2+\frac{1}{R^2}|b|^2\lesssim \mathcal{E}(\rho_0,\rho_1, T,R)+||\rho_0||_{p,R}^2+||\rho_1||_{p,R}^2.$$
\end{prop}
\begin{proof}

By rescaling we may assume $R=1$. For this proof, let $\mathcal{E}=\mathcal{E}(\rho_0,\rho_1, T, R=1)$. Let $\phi$ come from Lemma \ref{lag}. Define $b:=\nabla\phi(0)$, $A=\nabla^2\phi(0)$ and $B:=e^{-\frac{A}{2}}$. Clearly, $$\det(B)=\det(e^{-\frac{A}{2}})=e^{-\frac{tr(A)}{2}}=e^{-\frac{\Delta \phi}{2}}=e^0=1$$
where we use $\Delta\phi=0$. Using equation 2.3 from Lemma 2.1 in \cite{GO} and Lemma \ref{lag}
$$|B-Id|^2+|b|^2\lesssim \mathcal{E}(\rho_0,\rho_1, T,R)+||\rho_0||_{p,R}^2+||\rho_1||_{p,R}^2  .$$
Now, we have
\begin{equation*}
\begin{aligned}
\fint_{B_\theta}|\hat{T}-\hat{x}|^2\rho_0&=\fint_{B_\theta}|B(T(x)-b)-B^{-1}x|^2\rho_0\\
&= \fint_{BB_\theta}|T(x)-(B^{-2}x+b)|^2\rho_0\\
&\lesssim \fint_{B_{2\theta}}|T(x)-(B^{-2}x+b)|^2\rho_0\\
&\lesssim \fint_{B_{2\theta}}|T(x)-(x+\nabla\phi)|^2\rho_0+\fint_{B_{2\theta}}|(B^{-2}-Id-A)x|^2\rho_0\\
&+\fint_{B_{2\theta}}|\nabla\phi-b-Ax|^2\rho_0\\
& \lesssim \fint_{B_{2\theta}}|T(x)-(x+\nabla\phi)|^2\rho_0+ \theta^{2}|B^{-2}-Id-A|^2+\\ & \sup_{B_{2\theta}}|\nabla\phi -b-Ax|^2\\
&=\mathcal{I}+\mathcal{II}+\mathcal{III},
\end{aligned}
\end{equation*}
where 

$$\mathcal{I}:=\fint_{B_{2\theta}}|T(x)-(x+\nabla\phi)|^2\rho_0,$$
$$\mathcal{II}:=\theta^{2}|B^{-2}-Id-A|^2,$$ and
$$\mathcal{III}:=\sup_{B_{2\theta}}|\nabla\phi -b-Ax|^2.$$
Now as per Lemma \ref{lag}, we have the the first term 
$$\mathcal{I}\lesssim \theta^{-n}\left(\mathcal{E}^{\frac{n+3}{n+2}}+\gamma^2\right).$$
The second term can be estimated using Taylor series expansion 
$$|B^{-2}-Id-\nabla^2\phi(0)|^2=|e^A-Id-A|^2=\frac{|A|^4}{2}=\frac{|\nabla^2\phi(0)|^4}{2}.$$
Hence, $\mathcal{II}\lesssim \theta^{2} |\nabla^2\phi(0)|^4$.\\
Now ,
$$\sup_{B_{2\theta}}|\nabla\phi -b-Ax|^2=\sup_{B_{2\theta}}|f(x)-f(0)-\nabla f(0)x|^2,$$
where $f(x)=\nabla\phi(x)$. Use Taylor series expansion for $\nabla\phi$, we get 
$$\mathcal{III}\lesssim \sup_{B_{2\theta}}|x^T\nabla^2f(0)x|^2=\theta^{4}|\nabla^3\phi(0)|^2 .$$
Combining and multiplying by $\theta^{-2}$ on both sides,
\begin{equation*}
\begin{aligned}
\theta^{-2}\fint_{B_\theta}|\hat{T}-\hat{x}|^2\rho_0 &\lesssim
\theta^{-(n+2)}\mathcal{E}^\frac{n+2}{n+1}+\theta^{-(n+2)}||\rho_0||_{p,1}^2+\theta^{-(n+2)}||\rho_1||_{p,1}^2+|\nabla^2\phi(0)|^4+\theta^2|\nabla^3\phi(0)|^2.
\end{aligned}
\end{equation*}
Now using equation 2.3 from \cite{GO} and \eqref{lag2},  we have 

$$|\nabla^2\phi(0)|^4 \lesssim \left(\int_{B_1}|\nabla\phi|^2\right)^2 \lesssim \left(\mathcal{E}+\gamma^2\right)^2.$$
Similarly, we will get 
$$\theta^2|\nabla^3\phi(0)|^2 \lesssim \theta^2\left(\mathcal{E}+\gamma^2\right).$$
Hence, we have
\begin{equation*}
\begin{aligned}
 \theta^{-2}\fint_{B_\theta}|\hat{T}-\hat{x}|^2\rho_0 & \lesssim 
\theta^{-(n+2)}\mathcal{E}^{\frac{n+3}{n+2}}+\theta^{-(n+2)}\left(||\rho_0||_{p,1}^2+||\rho_1||_{p,1}^2\right) \\
& 
+\left(\mathcal{E}+||\rho_0||_{p,1}^2+||\rho_1||_{p,1}^2\right)^2+\theta^2\left(\mathcal{E}+||\rho_0||_{p,1}^2+||\rho_1||_{p,1}^2\right).
\end{aligned}
\end{equation*}
As we have assumed $\mathcal{E}+||\rho_0||_{p,1}^2+||\rho_1||_{p,1}^2\ll 1$, we can absorb second term of RHS on the third term($\mathcal{E}+||\rho_0||_{p,1}^2+||\rho_1||_{p,1}^2< \theta^2$ as $\theta$ to be fixed).
\begin{equation*}
\begin{aligned}
\theta^{-2}\fint_{B_\theta}|\hat{T}-\hat{x}|^2\rho_0 & \lesssim 
\theta^{-(n+2)}\mathcal{E}^{\frac{n+3}{n+3}}+\theta^{-(n+2)}\left(||\rho_0||_{p,1}^2+||\rho_1||_{p,1}^2\right)\\ &+\theta^2\left(\mathcal{E}+||\rho_0||_{p,1}^2+||\rho_1||_{p,1}^2\right).
\end{aligned}
\end{equation*}
Now we absorb the term with $\theta^2\left(||\rho_0||^2_{p,1}+||\rho_1||^2_{p,1}\right)$ in $\theta^{-(n+2)}\left(||\rho_0||_{p,1}^2+||\rho_1||^2_{p,1}\right)$ as $\theta$ to be chosen is small (less than 1):  
$$\theta^{-2}\fint_{B_\theta}|\hat{T}-\hat{x}|^2\rho_0\leq 
C\theta^{-(n+2)}\mathcal{E}^{\frac{n+3}{n+2}}+C\theta^{-(n+2)}\left(||\rho_0||_{p,1}^2+||\rho_1||_{p,1}^2\right)+C\theta^2\mathcal{E}.$$
Next, fix $\theta$ small so that $C\theta^2 \leq \frac{1}{2}\theta^{2\alpha'}$. This is possible because $\alpha'<1$. As we have $\mathcal{E}\ll 1$, conclude $C\theta^{-(n+2)}\mathcal{E}^\frac{n+3}{n+2}\leq \frac{1}{2}\theta^{2\alpha'}\mathcal{E}$. 
Thus we get $$\theta^{-2}\fint_{B_\theta}|\hat{T}-\hat{x}|^2\rho_0\leq \theta^{2\alpha'}\mathcal{E}+C_\theta \left(||\rho_0||_{p,1}^2+||\rho_1||_{p,1}^2\right).$$
\end{proof}


\pagebreak
\section{$\epsilon$ Regularity}
We now want to iterate our previous step to get $\epsilon$ regularity:
\begin{thm}\label{epsilon}
Let $||\rho_0||_{p,r}\leq \sqrt{\epsilon_1}r^\alpha$ for all $r>0$ and $$\mathcal{E}(\rho_0,\rho_1,T,R)+||\rho_0||_{p,R}^2+||\rho_1||_{p,R}^2\leq \epsilon_1$$
for some $\epsilon_1$ fixed. Then 
$$\sup_{r\leq \frac{R}{2}}\min_{A,b}\frac{1}{r^{2(1+\alpha)}}\fint_{B_r(0)}|T-(Ax+b)|^2\lesssim R^{-2\alpha}\left(\mathcal{E}(\rho_0,\rho_1,T,R)+||\rho_0||^2_{p,R}+||\rho_1||^2_{p,R}\right).$$
\end{thm}
\begin{proof}
Without loss of generality assume $R=1$.
Fix $\alpha'>\alpha$. Fix $0<\theta<\frac{1}{8}$. Define $\mathcal{E}_0=\mathcal{E}(\rho_0,\rho_1,T,1)$ and $\rho_i^0=\rho_i$ for $i=1,2$. Let $\epsilon$ be as in Proposition \ref{onestep}.\\

Applying Proposition \ref{onestep}, we know that there exists a symmetric matrix $B_1\in SL_{n}(\mathbb{R})$ and a vector $b_1$ such  that if we let $T_1(x):=B_1\left(T(B_1x)-b_1\right)$, $\rho_0^1(x):=\rho_0(B_1x)$ and $\rho_1^1(y)=\rho_1(B_1^{-1}y+b_1)$ then there exists $C_0>1$ satisfying

$$\mathcal{E}_1:=\mathcal{E}(\rho_0^0,\rho_1^0,T_1,\theta)\leq \theta^{2\alpha'}\mathcal{E}_0+C_0 p_0^2,$$
where $p_0^2=||\rho_0^0||^2_{p,1}+||\rho_1^0||_{p,1}^2$ .
Moreover $|B_1-Id|^2+|b_1|^2\leq C_0\left( \mathcal{E}_0+p_0^2\right)$.

Assuming we can keep applying Proposition \ref{onestep} repetitively, let $\rho_0^k$ and $\rho_1^k$ be the new source and target densities respectively after applying Proposition \ref{onestep} k times, that is, $\rho_0^k(x):=\rho_0^{k-1}(B_kx)$ and $\rho_1^{k}(x)=\rho_1^{k-1}\left(B_k^{-1}x+b_k\right),$ where $B_k\in SL_n(\mathbb{R})$ is symmetric and $b_k\in \mathbb{R}^n$. Let $T_k$ be the optimal transport map from $\rho_0^k$ and $\rho_1^k$. For all $k=0,1,\cdots$, define $\mathcal{E}_k:=\mathcal{E}(\rho_0^k,\rho_1^k, T_k, \theta^k)$ and for all $k=1,2,\cdots$, define $\Tilde{B_k}=|B_k-Id|$ and $p_k^2=||\rho_0^k||^2_{p,\theta^k}+||\rho_1^k||^2_{p,\theta^k}$. But we cannot necessarily apply Proposition \ref{onestep} as many times as we want. We need to justify that. Using Proposition \ref{onestep}, we know that there exists $\epsilon>0$ and $0<\theta<\frac{1}{8}$ such that for all $k\geq 1$ if we have 
\begin{equation}\label{iterationhypothesis}
    \begin{cases}
     \mathcal{E}_{k-1}+p_{k-1}^2\leq \epsilon\\
      B_1^{-1}(B_2^{-1}(\cdots B_{k-1}^{-1}(B_{k}^{-1}(B_{\theta^k})+b_k)+\cdots b_2)+b_1\subset B_\frac{1}{2},
    \end{cases}
\end{equation}
then 
\begin{equation}\label{iterationconclusion}
    \begin{cases}
     \mathcal{E}_{k}\leq \theta^{2\alpha'}\mathcal{E}_{k-1}+C_0p_{k-1}^2,\\
      p_k^2\leq C_0\theta^{2k\alpha}\epsilon_1\left(1+\Sigma_{i=1}^k\Tilde{B_i}\right)^4,\\
      \Tilde{B_k}^2\leq C_0\left(\mathcal{E}_{k-1}+p_{k-1}^2\right),\\
      b_k^2\leq C_0\theta^{2(k-1)}\left(\mathcal{E}_{k-1}+p_{k-1}^2\right).
    \end{cases}
\end{equation}
for some fixed constant $C_0>1$. The second condition of \eqref{iterationhypothesis} is important to justify $||\rho_1^k||_{p,\theta^k}=0$. Using definition of $\rho_1^i$ for $i=1,2,\cdots, k$ and using change of variables repetitively, we get 
\begin{equation*}
\begin{aligned}
||\rho_1^k||_{p,\theta^{k}}&=\left(\frac{1}{|B_{\theta^k}|}\int_{B_{\theta^{k}}}|\rho_1^{k-1}(B_k^{-1}x+b_k)-1|^pdx\right)^\frac{1}{p}\\
&= \left(\frac{1}{|B_{\theta^k}|}\int_{B_k^{-1}(B_{\theta^k})+b_k} |\rho_1^{k-1}(x)-1|^p dx\right)^\frac{1}{p}\\
&\vdots\\
&=\left(\frac{1}{|B_{\theta^k}|}\int_{B_{1}^{-1}(B_{2}^{-1}\cdots(B_{k-1}^{-1}(B_k^{-1}(B_{\theta^k})+b_k)+b_{k-1})+\cdots)b_2)+b_1}|\rho_1(x)-1|^pdx\right)^\frac{1}{p}\\
&=0.
\end{aligned}
\end{equation*}
In the last step above, we use condition 2 from \eqref{iterationhypothesis} and $\rho_1=1$ in $B_\frac{1}{2}$. Using this, we justify second condition in \eqref{iterationconclusion}, which was not obvious from Proposition \ref{onestep}:\\

\begin{equation*}
\begin{aligned}
p_k:=||\rho_0^k||_{p,\theta^k}&=\left(\frac{1}{|B_{\theta^k}|}\int_{B_{\theta^k}}|\rho_0^{k-1}(B_k x)-1|^p dx\right)^\frac{1}{p}\\
&=\left(\frac{1}{|B_{\theta^k}|}\int_{B_k(B_{\theta^k})}|\rho_0^{k-1}(x)-1|^p dx\right)^\frac{1}{p}\\
& \vdots\\
&=\left(\frac{1}{|B_{\theta^k}|}\int_{B_1(B_2(\cdots (B_k(B_{\theta^k}))\cdots)}|\rho_0(x)-1|^p dx\right)^\frac{1}{p}.\\
\end{aligned}
\end{equation*}
For a linear map $A$, we know $A(B_r)\subset B_{r(1+|A-Id|)}$. Using this we get,
$$B_1(B_2(\cdots (B_k(B_{\theta^k}))\cdots)\subset B_{\theta^kr_k},$$
where $r_k:=(1+|B_1-Id|)(1+|B_2-Id|)\cdots (1+|B_k-Id|).$
Thus, 
\begin{equation*}
\begin{aligned}
p_k&\leq \left(\frac{1}{|B_{\theta^k}|}\int_{B_{\theta^kr_k}}|\rho_0(x)-1|^p dx\right)^\frac{1}{p}\\
&=\left(\frac{|B_{\theta^kr_k}|}{|B_{\theta^k}|}\fint_{B_{\theta^kr_k}}|\rho_0-1|^p\right)^\frac{1}{p}\\
&\leq \Pi_{i=1}^k\left(1+|B_i-Id|\right)^{\frac{n}{p}+\alpha}\sqrt{\epsilon_1}\theta^{k\alpha}.
\end{aligned}
\end{equation*}
As $p>n$ and $0<\alpha<1$, we have 
$$p_k^2\leq \Pi_{i=1}^k\left(1+\Tilde{B_i}\right)^{4}\epsilon_1\theta^{2k\alpha}. $$
We claim that if we let 
\begin{equation}
\begin{aligned}
     C_2&=C_0(8+\epsilon)\\
     C_1&=1+C_0C_2\theta^{-2\alpha}\sum_{i=1}^{\infty} \theta^{2i(\alpha'-\alpha)}\\
     C_3=C_4&=C_0(C_1+C_2)\theta^{-2\alpha}
\end{aligned}
\end{equation}
and provided 
\begin{equation}
\left[(C_1+C_2)+8C_3^2\left(\sum_1^\infty \theta^{i\alpha}\right)^4+\sqrt{C_4}\sum_1^\infty \theta^{i\alpha}\right]\sqrt{\epsilon_1}<\epsilon<\frac{1}{8},
\end{equation}
then we have 
\begin{equation}\label{geometricdecay}
\begin{aligned}
\mathcal{E}_k&\leq C_1\epsilon_1\theta^{2k\alpha}\\
p_k^2&\leq C_2\epsilon_1\theta^{2k\alpha}\\
\Tilde{B}_{k+1}^2&\leq C_3\epsilon_1\theta^{2k\alpha}\\
b_{k+1}^2&\leq C_4\epsilon_1\theta^{2k\alpha},
\end{aligned}
\end{equation}
for all $k$.
To prove this claim, we use strong induction. Our hypothesis in this theorem says $$\mathcal{E}_0+p_0^2\leq \epsilon_1.$$ Thus we can apply Proposition \ref{onestep} to get $\Tilde{B_1}^2\leq C_0\epsilon_1<C_3\epsilon_1$ and $b_1^2\leq C_0\epsilon_1<C_4\epsilon_1$.
Assume \eqref{geometricdecay} is true for $k\leq N$. We show that under this assumption we can apply \eqref{iterationconclusion} for $k=N+1$. We need to check \eqref{iterationhypothesis} for $k=N+1$. Firstly, 
\begin{equation*}
\begin{aligned}
\mathcal{E}_N+p_N^2& \leq C_1\epsilon_1\theta^{2N\alpha}+C_2\epsilon_1\theta^{2N\alpha}\\
&\leq (C_1+C_2)\epsilon_1\\
&\leq \epsilon.
\end{aligned}
\end{equation*}
Next, we want to show 
$$B_1^{-1}(B_2^{-1}(\cdots B_{N}^{-1}(B_{N+1}^{-1}(B_{\theta^{N+1}})+b_{N+1})+\cdots b_2)+b_1\subset B_\frac{1}{2}.$$
It is enough to justify 
$$|b_1|+|B_1^{-1}||b_2|+|B_1^{-1}||B_2^{-1}||b_3|+\cdots+|B_1^{-1}||B_2^{-1}|\cdots|B_N^{-1}||b_{N+1}|$$
$$+|B_1^{-1}||B_2^{-1}|\cdots|B_{N+1}^{-1}|\theta^{N+1}\leq\frac{1}{2}.$$
Now, use $|B_i^{-1}|\leq 1+2\Tilde{B_i}$ to get  
\begin{equation*}
\begin{aligned}
&|b_1|+|B_1^{-1}||b_2|+|B_1^{-1}||B_2^{-1}||b_3|+\cdots+|B_1^{-1}||B_2^{-1}|\cdots|B_N^{-1}||b_{N+1}|\\&+|B_1^{-1}||B_2^{-1}|\cdots|B_{N+1}^{-1}|\theta^{N+1}\\
&\leq|b_1|+(1+2\Tilde{B_1})|b_2|+\cdots +(1+2\Tilde{B_1})\cdots (1+2\Tilde{B_N})|b_{N+1}|\\&+(1+2\Tilde{B_1})\cdots (1+2\Tilde{B_{N+1}})\theta^{N+1}\\
&\leq \sum_{i=1}^{N+1}|b_i|+4\sum_{i=1}^{N}\Tilde{B_i}\sum_{1}^{N+1}|b_j|+\theta^{N+1}+4\sum_{1}^{N+1}\Tilde{B_i}\theta^{N+1}.        
\end{aligned}
\end{equation*}
Recall we have \eqref{geometricdecay} for $k\leq N$. We use this to show each of the four term above is smaller than $\frac{1}{8}$:
$$\sum_1^{N+1}|b_i|\leq \sqrt{C_4\epsilon_1}\sum_1^{N+1}\theta^{i\alpha}\leq \epsilon<\frac{1}{8},$$
$$4\sum_{i=1}^{N}\Tilde{B_i}\sum_{1}^{N+1}|b_j|\leq 4\cdot \frac{1}{8}\sqrt{C_3\epsilon_1}\sum_1^N \theta^{i\alpha}\leq \frac{1}{2}\epsilon<\frac{1}{16},$$
$$\theta^{N+1}<\frac{1}{8},$$ and
$$4\sum_{1}^{N+1}\Tilde{B_i}\theta^{N+1}<4\epsilon\frac{1}{8}<\frac{1}{8}.$$
This shows we can justify the second condition of \eqref{iterationhypothesis} for $k=N+1$. This means we can apply \eqref{iterationconclusion} for $k\leq N+1$.\\
Now, let $\alpha'=\alpha+\delta$ for some $\delta>0$. Then,
\begin{equation*}
\begin{aligned}
\mathcal{E}_{N+1}&\leq \theta^{2\alpha'}\mathcal{E}_{N}+C_0 p_N^2\\
&\leq \theta^{2\alpha'}\left(\theta^{2\alpha'}\mathcal{E}_{N-1}+C_0p_{N-1}^2\right)+C_0 p_N^2\\
&\leq \theta^{4\alpha'}\mathcal{E}_{N-1}+C_0(P_N^2+\theta^{2\alpha'}P_{N-1}^2)\\
&\vdots\\
&\leq \theta^{2(N+1)\alpha'}\mathcal{E}_0+C_0\sum_{0}^{N}\theta^{2(N-i)\alpha'}p_i^2\\
&\leq \theta^{2(N+1)\alpha}\epsilon_1+C_0\sum_{0}^{N}\theta^{2(N-i)(\alpha+\delta)}C_2\epsilon_1\theta^{2i\alpha}\\
&\leq \epsilon_1\theta^{2(N+1)\alpha}+C_0C_2\epsilon_1\theta^{2N\alpha}\sum_0^N\theta^{2(N-i)\delta}\\
&\leq \epsilon_1\theta^{2(N+1)\alpha}\left(1+C_0C_2\theta^{-2\alpha}\sum_0^\infty\theta^{2i(\alpha'-\alpha)}\right),\\
&=C_1\epsilon_1\theta^{2(N+1)\alpha}.
\end{aligned}
\end{equation*}
where we use $\alpha<\alpha'$ and the first inequality in \eqref{iterationconclusion} for $k=1,2,\cdots, N+1$ repetitively. We have also used $p_i^2\leq C_2\epsilon_1\theta^{2i\alpha}$ for $i=1, \cdots,N$.\\
Next, using the second condition of \eqref{iterationconclusion} for $k=N+1$ and \eqref{geometricdecay} for $k=1,\cdots, N$, we get
\begin{equation*}
\begin{aligned}
p_{N+1}^2&\leq C_0\theta^{2(N+1)\alpha}\epsilon_1\left(1+\sum_{0}^{N+1}\Tilde{B_j}\right)^4\\
&\leq C_0\theta^{2(N+1)\alpha}\epsilon_1\left(1+\sqrt{C_3\epsilon_1}\sum_0^{N+1}\theta^{j\alpha}\right)^4\\
&\leq 8C_0\theta^{2(N+1)\alpha}\epsilon_1\left(1+C_3^2\epsilon_1^2\left(\sum_0^\infty\theta^{j\alpha} \right)^4\right)\\
&\leq C_0(8+\epsilon)\epsilon_1\theta^{2(N+1)\alpha}\\
&=C_2\epsilon_1\theta^{2(N+1)\alpha}.
\end{aligned}
\end{equation*}
We now want to show geometric decay for $\Tilde{B}_{N+2}^2$ and $|b_{N+2}|^2$. But we cannot yet apply the third or fourth inequality in  \eqref{iterationconclusion} for $k=N+1$. We already have established $\mathcal{E}_{N+1}\leq C_1\epsilon_1\theta^{2(N+1)\alpha}$ and $p_{N+1}^2\leq C_2 \epsilon_1\theta^{2(N+1)\alpha}$. Thus, 
\begin{equation*}
\begin{aligned}
\mathcal{E}_{N+1}+p_{N+1}^2& \leq (C_1+C_2)\epsilon_1\theta^{2(N+1)\alpha}\\
&\leq \epsilon.
\end{aligned}
\end{equation*}
Now we can apply Proposition \ref{onestep}, to get
\begin{equation*}
\begin{aligned}
\Tilde{B}_{N+2}^2&\leq C_0\left( \mathcal{E}_{N+1}+p_{N+1}^2\right)\\
&\leq C_0 (C_1+C_2)\theta^{-2\alpha}\epsilon_1\theta^{2(N+2)\alpha}\\
&=C_3\epsilon_1\theta^{2(N+2)\alpha}.
\end{aligned}
\end{equation*}
Similarly, 
\begin{equation*}
\begin{aligned}
|b_{N+2}|^2&\leq \theta^{2(N+1)}C_0\left( \mathcal{E}_{N+1}+p_{N+1}^2\right)\\
&\leq C_0 (C_1+C_2)\theta^{-2\alpha}\epsilon_1\theta^{2(N+2)\alpha}\\
&=C_4\epsilon_1\theta^{2(N+2)\alpha}.
\end{aligned}
\end{equation*}
Thus we have shown \eqref{geometricdecay} holds for $k\leq N+1$. This gives $\mathcal{E}_k\leq C_1\epsilon_1\theta^{2k\alpha}$ for all $k$.

\begin{equation*}
\begin{aligned}
    |B_kB_{k-1}\cdots B_1-Id|&=\left|\left((B_k-Id)+Id\right)\left((B_{k-1}-Id)+Id\right)\cdots \left((B_{1}-Id)+Id\right)-Id\right|\\
    &\leq \Pi_{1}^k(1+\tilde{B}_i)-1\\
    &\leq 4\sum_1^k \tilde{B}_i.
\end{aligned}
\end{equation*}

Now using this and geometric decay of energy $\mathcal{E}_k$ we are done (details follow similarly to the argument after equation (3.55) in \cite{GO}).

\end{proof}

\section{Proof of Main Theorem}
From Theorem \ref{epsilon}, we have a good linear approximation of $T$ in $B_{\frac{1}{2}}$. Recall that $T=\nabla u$ for some convex function $u$. Thus there exists a linear function $L$ such that for $r<\frac{1}{2}$, we have 
$$\fint_{B_r}|\nabla u-L|^2\lesssim r^{2(1+\alpha)}.$$ 
Use the Poincaré inequality to get a quadratic polynomial $Q$ satisfying $\fint_{B_r}|u-Q|^2\lesssim r^{2(2+\alpha)}$ for all $r<\frac{1}{2}$. Up to an affine transformation, $Q(x)=\frac{|x|^2}{2}$. Define $\Tilde{u}(x):=\frac{u(rx)}{r^2}$ to get $$\fint_{B_1}|\Tilde{u}(x)-Q(x)|^2\lesssim r^{2\alpha}.$$ From here, we aim to derive an $L^\infty$ estimate. Define $\delta>0$, such that $\delta^2=r^{2\alpha}$. We now apply Lemma \ref{pointwise} to get $C^{2,\frac{4\alpha}{4+n}}$ regularity for $u$ at the origin. Recalling that $\mu = \frac{4}{4+n}$ we have
$$| \Tilde{u}-Q|\lesssim \delta^\mu=r^{\mu\alpha}\quad\text{in}\; B_\frac{1}{2}$$
$$\implies \left|\frac{u(rx)}{r^2}-\frac{Q(rx)}{r^2}\right|^2\lesssim r^{2\mu\alpha} \quad\text{in}\; B_\frac{1}{2}$$
$$\implies |u-Q| \lesssim r^{2+\alpha\mu} \quad \text{in}\, B_{r/2}.$$
This means $C^{2,\frac{4\alpha}{4+n}}$ regularity at the origin. This completes the proof of Theorem \ref{Main}.

\section{Sharpness}\label{sharp}
Here we will show that we cannot improve the $L^\infty$ estimate in Theorem \ref{Main} starting from the $L^2$ estimate we have. We demonstrate this via an explicit example. Let $Q = |x|^2/2$. We construct a function that satisfies 
\begin{equation}\label{ca1}
\fint_{B_r}|u-Q|^2\lesssim r^{2(2+\alpha)}\quad\text{in}\; B_r,\, 0 < r < 1
\end{equation}
but $||u-Q||_{L^\infty(B_\frac{r}{2})}\lesssim r^{2+\beta},\, r \in (0,\,1)$ is not true for any $\beta>\frac{4\alpha}{4+n}.$\\
Now construct $u$ as below:\\
For $k \geq 0$, inside $B_{2^{-k}}(0),$ replace $Q$ by a linear map in a small ball $B_k$ of radius $r_k := 2^{-k\left(1+\frac{2\alpha}{n+4}\right)}/100$ tangent to $\partial B_{2^{-k}}$. These linear parts won't intersect each other due to the size of the portions on which we are replacing $Q$ by planes. It is straightforward to check that $\int_{B_k} |u-Q|^2 \lesssim r_k^{4+n}$, hence
$$\fint_{B_{2^{-N}}} |u-Q|^2 \lesssim 2^{nN} \sum_{k=n}^{\infty} r_k^{4+n} \lesssim \left(2^{-N}\right)^{2(2+\alpha)}$$
for all $N$. Inequality (\ref{ca1}) follows.\\
On the other hand, in $B_k,$ we have $\max (u-Q) = r_k^2/2$, thus
$$\|u-Q\|_{L^{\infty}(B_{2^{-N-1}})} \geq r_{N+1}^2/2 = c(n,\,\alpha) (2^{-N})^{2+\frac{4\alpha}{4+n}}$$
for all $N$. The right-hand side clearly dominates $(2^{-N})^{(2+\beta)}$ for any $\beta > \frac{4\alpha}{4+n}.$ This shows that we cannot get a better $L^\infty$ estimate.\\



\section*{Acknowledgments}
The author would like to thank Prof. Connor Mooney for his supervision and feedback on the paper. The author was supported by the Sloan Fellowship, UC Irvine Chancellor's Fellowship, and NSF CAREER Grant DMS-2143668 of C. Mooney.

\end{document}